\documentclass[12pt,reqno]{amsart}
\usepackage[margin=1.2in]{geometry}
\usepackage{amssymb}
\usepackage{amscd}
\usepackage{amsfonts}
\usepackage{setspace}
\usepackage{mathrsfs}
\usepackage{version}

\usepackage{graphicx}

\newtheorem{theorem}{Theorem}[section]
\newtheorem{lemma}[theorem]{Lemma}
\newtheorem{proposition}[theorem]{Proposition}

\theoremstyle{definition}

\renewcommand{\leq}{\leqslant}
\renewcommand{\geq}{\geqslant}

\newcommand\sml{\operatorname{small}}

\def\R{\mathbb{R}}

\def\Z{\mathbf{Z}}

\def\N{\mathbf{N}}

\def\eps{\varepsilon}

\numberwithin{equation}{section}

\begin{document}

\title[Extremal problems for GCDs]{Extremal problems for GCDs}


%    author one information

\author[Green]{Ben Green}
\address{Mathematical Institute \\ Andrew Wiles Building \\ Radcliffe Observatory Quarter \\ Woodstock Rd \\ Oxford OX2 6QW}
\email{ben.green@maths.ox.ac.uk}

\author[Walker]{Aled Walker}
\address{Trinity College \\ Trinity Street \\ Cambridge CB2 1TQ}
\email{aw530@cam.ac.uk}

\thanks{The first-named author is supported by a Simons Investigator Award and is grateful to the Simons Foundation for their support. The second-named author is supported by a Postdoctoral Fellowship with the Centre de Recherches Math\'{e}matiques and by a Junior Research Fellowship from Trinity College Cambridge. We would also like to thank Andrew Granville, Dimitris Koukoulopoulos, and James Maynard for interesting conversations related to this work. }

%\onehalfspace
%    \subjclass is required.
\subjclass[2000]{Primary }

\begin{abstract}
We prove that if $A \subseteq [X, 2X]$ and $B \subseteq [Y, 2Y]$ are sets of integers such that $\gcd(a,b) \geq D$ for at least $\delta |A||B|$ pairs $(a,b) \in A \times B$ then $|A||B| \ll_{\eps} \delta^{-2 - \eps} XY/D^2$. This is a new result even when $\delta = 1$. The proof uses ideas of Koukoulopoulos and Maynard and some additional combinatorial arguments.
\end{abstract}
%    Abstract is required.
\maketitle

\section{Introduction and proof strategy}

Fix $\eps \in (0,1)$ throughout the paper; implied constants and thresholds may depend on $\eps$ but are otherwise absolute. Let $p_0$ be a threshold which (at the start of Section \ref{sec4}) will be taken to be sufficiently large. Given a finite set $S \subset \N$, write $\mathscr{P}(S)$ for the set of primes dividing some element of $S$, and $\mathscr{P}_{\sml}(S)$ for the set of primes $p \leq p_0$ dividing some element of $S$. 

Our main result is the following.

\begin{theorem}
\label{mainthm} Let $X,Y, D \in [1,\infty)$, and suppose that $D \leq \min(X,Y)$. Let $\delta \in (0,1]$. Suppose that $A \subset [X, 2X]$ and $B \subset [Y, 2Y]$ have the following property: for at least $\delta |A| |B|$ pairs $(a,b) \in A \times B$, $\gcd(a,b) \geq D$. Then we have the bound
\[ |A| |B| \leq (1000)^{1+ \# \mathscr{P}_{\sml}(A \cup B)} \delta^{-2 - \eps} \frac{XY}{D^2} .\]
\end{theorem}

Let us make some remarks on this theorem. 

\emph{1.} This obviously implies the cruder bound $|A||B| \ll \delta^{-2 - \eps} XY/D^2$, mentioned in the abstract. The more precise form we have stated seems of little additional interest in its own right, but is critical for the proof. Perhaps the most natural case is when $A = B$ and $X = Y$, when the result says the following: if, for a proportion $\delta$ of all pairs $(a,a') \in A \times A$ we have $\gcd(a, a') \geq D$, then $|A| \ll \delta^{-1-\eps} X/D$. 

\emph{2.} We believe that the result is new even when $\delta = 1$, that is to say when $\gcd(a,b) \geq D$ for \emph{all} $a \in A$ and $b \in B$. In this case, the result is clearly sharp up to a multiplicative constant. Indeed, assuming that $D$ is an integer, we may take $A = \{ x \in [X, 2X] : D | x\}$ and $B = \{y \in [Y, 2Y] : D | y\}$. One might wonder whether all the tight examples have approximately such a structure; however, Chow has constructed a different family of tight examples for $\delta = 1$ (see \cite[Section 15]{km}) in which there is no single $d \gg D$ that divides a positive proportion of $A$ and $B$. 

\emph{3.} For $\delta \in (0,1)$ the result is also sharp for a wide range of parameters, apart from the factor of $\delta^{-\eps}$. To see this, let $D \geq \delta^{-1}$ be given, set $D_0 := \lfloor \delta D\rfloor$, and consider the sets $A = B = \{ x \in [X, 2X] : D_0 | x\}$. Evidently $|A||B| \sim \delta^{-2} X^2/D^2$. However, if $x = D_0 m, x' = D_0 m'$ with $m, m' \in [X/D_0, 2X/D_0]$ and $\gcd(m, m') \geq D/D_0$ then $x, x' \in A$ and $\gcd(x,x') \geq D$.  The proportion of pairs of integers with gcd $k$ is $1/k^2 \zeta(2)$, and so the proportion of pairs of integers with gcd $\geq k$ is $\gg 1/k$. It follows (at least if $X/D_0$ is big enough compared to $D/D_0$) that the number of such pairs $(x, x')$ is $\gg \delta |A||B|$.

\emph{4.} When $A = B$, $X = Y$ and $\delta = 1$, the result says the following: if $\gcd(a,a') \geq D$ for all $a, a' \in A$, then $|A| \ll X/D$. However (we are rather embarrassed to admit) Zachary Chase pointed out to the authors that this particular result is trivial, because the assumption implies that $|a - a'| \geq D$ whenever $a \neq a'$. This argument does not, however, appear to extend to the other cases.

\emph{5.} A straightforward dyadic decomposition argument would allow one to establish similar results under the assumption that $A \subset [X]$ and $B \subset [Y]$. We leave the details to the reader.\\

\emph{Notation.} Our notation is standard. If $p$ is a prime and $a \in \Z$, we write $v_p(a)$ for the largest $k$ such that $p^k | a$. We extend this to rationals by $v_p(a/b) = v_p(a) - v_p(b)$. Implied constants in the $O(), \ll$ and $\gg$ notations are absolute (though they may depend on $\eps$, which is fixed throughout the paper).\vspace*{10pt}

\emph{Strategy.} Our strategy for proving Theorem \ref{mainthm} is essentially to proceed by induction on $\#\mathscr{P}(A \cup B)$, but we will phrase the argument in terms of a hypothetical counterexample with minimal $\# \mathscr{P}(A \cup B)$. The first main business is to show that such a minimal counterexample has a very specific structure.

\begin{proposition}\label{minimal-counterexample}
Suppose we have a counterexample to Theorem \ref{mainthm} with the set $\mathscr{P}(A \cup B)$ minimal in size. Let $\Omega \subset A \times B$, with $|\Omega| = \delta |A| |B|$, be the set of pairs for which $\gcd(a,b) \geq D$. Then there is $\Omega' \subset \Omega$, with $|\Omega' | \geq \frac{1}{2} |\Omega|$, and an integer $N$ such that the following is true. For all primes $p$ and for all $(a,b) \in \Omega'$ we have $|v_p(a/N)| + | v_p(b/N)| \leq 1$.\end{proposition}

Though such a statement does not appear explicitly in their work, this proposition should be considered essentially due to Koukoulopoulos and Maynard \cite{km}. We will give a fairly short, self-contained proof. On
some level this is equivalent to the argument of \cite{km}, but we phrase things
rather differently.  

To complete the proof of Theorem \ref{mainthm}, we prove the following counterpart to Proposition \ref{minimal-counterexample}.

\begin{proposition}\label{not-counterexample}
Suppose that $A \subset [X,2X]$, $B \subset [Y, 2Y]$, $\Omega, D,\delta, N$ are as in Proposition \ref{minimal-counterexample}. Then $|A||B| \leq 1000 \delta^{-2} XY/D^2$.
\end{proposition}
Evidently, this means that $A, B$ do not in fact give a counterexample to Theorem \ref{mainthm}. Combining Propositions \ref{minimal-counterexample} and \ref{not-counterexample} shows that no minimal counterexample to Theorem \ref{mainthm} exists, so Theorem \ref{mainthm} is true.

The proof of Proposition \ref{not-counterexample} uses some combinatorial arguments and is not found in \cite{km}.

\section{Concentrated measures on $\Z^2$}

In this section we prove a result about concentration of probability measures on $\Z^2$. It is the key technical ingredient in the proof of Proposition \ref{minimal-counterexample}, where it is used to concentrate the pair of valuation functions $(v_p(a),v_p(b))$ around a diagonal pair $(k,k)$. 

Here, as in the rest of the paper, we write $q = 2 + \eps$ and write $q'$ for the conjugate index to $q$ (i.e. $\frac{1}{q} + \frac{1}{q'} = 1$). 

\begin{lemma}\label{conc-lem}
Let $c \leq 1$ and $\lambda \leq \frac{4}{5}$. Suppose that $\mu$ is a finitely-supported probability measure on $\Z^2$. Suppose that there are sequences $x = (x_i)_{i \in \Z}$, $y = (y_j)_{j \in \Z}$ of non-negative reals such that $\Vert x \Vert_{\ell^{q'}(\Z)} = \Vert y \Vert_{\ell^{q'}(\Z)} = 1$, and such that for all $(i,j) \in \Z^2$ we have
\begin{equation}\label{key-bd} \mu(i,j) \leq c\lambda^{|i - j|} x_i y_j.\end{equation}
Then $c \geq \frac{1}{9}$, and $\mu$ is highly concentrated near some point $(k,k)$:
\begin{equation}\label{concentrate} \sum_{|i - k| + |j - k| \geq 2} \mu(i,j) \ll \lambda^{q + \eps}.\end{equation} 
\end{lemma}
\begin{proof}
We first prove the lower bound on $c$. Using \eqref{key-bd}, $\sum_{\Z^2} \mu(i,j) = 1$, $\sum_{m \in \Z} \lambda^{|m|} \leq 9$ and $q'  < 2$, we have
\[ \frac{1}{c} \leq \sum_{i,j} \lambda^{|i - j|}x_i y_j \leq \sum_{i,j} \lambda^{|i - j|} (x_iy_j)^{\frac{q'}{2}} \leq 9 \sup_m \sum_i (x_i y_{i+m})^{\frac{q'}{2}} \leq 9,\]
where the last step follows from the Cauchy-Schwarz inequality and the assumption $\Vert x \Vert_{\ell^{q'}(\Z)} = \Vert y \Vert_{\ell^{q'}(\Z)} = 1$. The lower bound on $c$ follows. 

Turning to \eqref{concentrate}, write $\sup_i x_i y_i = 1 - \gamma$ for some $\gamma\in [0,1]$, and suppose this supremum attained when $i = k$. Then $x_k, y_k \geq 1 - \gamma$, so 
\begin{equation}\label{key-ineq}  \sum_{i \neq k} x^{q'}_i, \sum_{j \neq k} y^{q'}_j \ll \gamma \; \mbox{and} \; x_i, y_j \ll \gamma^{1/q'} \; \mbox{when} \; i,j \neq k.\end{equation}
For $n = 1,2,3,4,5,6$ write $\Sigma_n := \sum_{(i,j) \in S_n} \mu(i,j)$, where $S_1,\dots, S_6$ are the following sets, which partition $\Z^2$:
\[ S_1 := \{(i,j)\in \Z^2 : i \neq j \neq k \neq i\}, \quad S_2 := \{(k, j) : |j - k| \geq 2\},\]
\[  S_3 := \{(i, k) : |i - k| \geq 2\}, \quad S_4 := \{(k, k \pm 1), (k \pm 1, k)\}, \quad S_5 := \{(i,i) : i  \neq k\}, \]  and finally $S_6 := \{(k,k)\}$. We bound the $\Sigma_n$ in turn. \vspace*{8pt}

\emph{Bound for $\Sigma_1$.} By \eqref{key-bd} and \eqref{key-ineq} we have
\[ \Sigma_1 \leq \sum_{\substack{ i,j \neq k \\ i\neq j}} \lambda^{|i - j|} x_i y_j \ll \gamma^{\frac{2}{q'} (1 - \frac{q'}{2})} \sum_{\substack{ i,j \neq k \\ i\neq j}} \lambda^{|i - j|} (x_iy_j)^{\frac{q'}{2}} =  \gamma^{\frac{2}{q'} -1} \sum_{m \neq 0} \lambda^{|m|}\!\! \!\! \sum_{i,i+m \neq k}  \!\!\! (x_i y_{i + m})^{\frac{q'}{2}}.  \] By Cauchy-Schwarz and \eqref{key-ineq}, for each fixed $m$ we have
\[ \sum_{i,i+m \neq k}  (x_i y_{i + m})^{\frac{q'}{2}} \leq (\sum_{i \neq k} x_i^{q'})^{1/2} (\sum_{j \neq k} y_j^{q'})^{1/2} \ll \gamma.\]
Since $\sum_{m \neq 0} \lambda^{|m|} \ll \lambda$, putting these together gives $\Sigma_1 \ll \lambda \gamma^{\frac{2}{q'}}$.\vspace*{8pt}

\emph{Bounds for $\Sigma_2,\Sigma_3$.} For $\Sigma_2$, we use the trivial bound $x_k \leq 1$ and \eqref{key-ineq} for $y_j$. This gives (using the assumption that $\lambda \leq \frac{4}{5}$)
\[ \Sigma_2 \leq \sum_{|j - k| \geq 2} (5\lambda/4)^{|j - k|} (4/5)^{|j-k|} y_j \ll \lambda^2 \!\!\!\sum_{|j - k| \geq 2} \!\!\! (4/5)^{|j - k|} y_j \ll \gamma^{\frac{1}{q'} (1 - \frac{q'}{2})}\lambda^2\!\!\! \sum_{|j - k| \geq 2} \!\!\! (4/5)^{|j - k|} y_j^{\frac{q'}{2}} .\]
By Cauchy-Schwarz and \eqref{key-ineq}, 
\[ \sum_{|j - k| \geq 2}  (4/5)^{|j - k|} y_j^{\frac{q'}{2}} \ll (\sum_{j \neq k} (4/5)^{2|j - k|})^{1/2} (\sum_{j \neq k} y_j^{q'})^{1/2} \ll \gamma^{1/2}.\]
Combining these bounds gives $\Sigma_2 \ll \gamma^{\frac{1}{q'}} \lambda^2$, and an essentially identical argument yields $\Sigma_3 \ll \gamma^{\frac{1}{q'}} \lambda^2$.\vspace*{8pt}

\emph{Bound for $\Sigma_4$.} From \eqref{key-bd}, \eqref{key-ineq} we immediately get $\Sigma_4 \ll \gamma^{\frac{1}{q'}} \lambda$.\vspace*{8pt}

\emph{Bound for $\Sigma_5$.} A trivial modification to the argument used for $\Sigma_1$ (allowing $i = j$, which gives just a term with $m = 0$) shows that $\Sigma_5 \ll \gamma^{\frac{2}{q'}}$.\vspace*{8pt}

\emph{Bound for $\Sigma_5 + \Sigma_6$.} By \eqref{key-bd} and the fact that $\sup_i x_i y_i = 1 - \gamma$, 
\[ \Sigma_5 + \Sigma_6 \leq \sum_i x_i y_i \leq (1 - \gamma)^{1 - \frac{q'}{2}} \sum_i (x_i  y_i)^{\frac{q'}{2}} \leq (1 - \gamma)^{1 - \frac{q'}{2}} \leq 1 - (1 - \frac{q'}{2})\gamma, \] where we used Cauchy-Schwarz yet again.\vspace*{8pt}

Putting all this together gives
\[ 1 = \sum_{n = 1}^6 \Sigma_n \leq 1 - (1 - \frac{q'}{2})\gamma + O(\lambda \gamma^{\frac{1}{q'}}) .\] This implies that $\lambda \gg \gamma^{1 - \frac{1}{q'}}$, i.e. $\gamma \ll \lambda^{q}$.  
Finally, we see that 
\[ \sum_{|i - k| + |j - k| \geq 2} \mu(i,j) = \Sigma_1 + \Sigma_2 + \Sigma_3 + \Sigma_5 \ll \lambda^{\frac{2q}{q'}} +\lambda^{2 + \frac{q}{q'}} \ll \lambda^{\frac{2q}{q'}} \] (since $q \leq 3$).
The result follows, noting that $\frac{2q}{q'} = q + \eps$. \end{proof}

\section{Properties of a minimal counterexample}\label{sec4}

We turn now to the proof of Proposition \ref{minimal-counterexample}. We first reduce matters to the following ``local'' statement at a single prime $p$.

\begin{proposition}\label{minimal-counterexample-p}
Suppose we have a counterexample to Theorem \ref{mainthm} with the set $\mathscr{P}(A,B)$ minimal in size. Let $\Omega \subset A \times B$, with $|\Omega| = \delta |A| |B|$, be the set of pairs for which $\gcd(a,b) \geq D$. Let $p \in \mathscr{P}(A,B)$ be a prime. Then $p > p_0(\eps)$, and there is $k_p \in \Z_{\geq 0}$ and $\Omega_p \subset \Omega$ such that for all $(a,b) \in \Omega_p$ we have $|v_p(a) - k_p| + | v_p(b) - k_p| \leq 1$, and such that $|\Omega \setminus \Omega_p| \ll  p^{-1 - \eps/3 } |\Omega|$ .
\end{proposition}

Proposition \ref{minimal-counterexample} follows quickly from this by taking $N = \prod_p p^{k_p}$ and $\Omega' := \bigcap_p \Omega_p$. We have
\[ |\Omega'| \geq |\Omega| \big(1 - O \big(\sum_{p > p_0} p^{-1 - \eps/3}\big)\big) \geq \frac{1}{2}|\Omega|\] if $p_0$ is big enough (this is the point at which $p_0$ is constrained).

It remains, then, to establish Proposition \ref{minimal-counterexample-p}. Fix, for the rest of this section, the prime $p$. For $i,j \in \Z_{\geq 0}$, we define $A_i := \{a \in A : v_p(a) = i\}$, $B_j := \{ b \in B : v_p(b) = j\}$, and write $\alpha_i := \frac{|A_i|}{|A|}$ and $\beta_j := \frac{|B_j|}{|B|}$ for the relative densities of these sets. Write $\mu(i,j) := \frac{|\Omega \cap (A_i \times B_j)|}{|\Omega|}$, thus $\mu$ is a finitely-supported probability measure on $\Z_{\geq 0}^2$.

For any $i,j$, consider the sets $\bar{A}_i := p^{-i} \cdot A_i$ and $\bar{B}_j := p^{-j}\cdot B_j$. These are sets of integers, coprime to $p$, with $\bar{A}_i \subset [\frac{X}{p^i}, \frac{2X}{p^i}]$, $\bar{B}_j \subset [\frac{Y}{p^j}, \frac{2Y}{p^j}]$ and $\gcd(x, y) \geq \frac{D}{p^{\min(i,j)}}$ whenever $x = p^{-i}a$, $y = p^{-j} b$ with $(a,b) \in \Omega$. 

By the minimality assumption, these sets cannot be a counterexample to Theorem \ref{mainthm}, and therefore we have the inequality
\begin{equation}\label{to-comp-1} |\bar{A}_i||\bar{B}_j| \leq (1000)^{1+ \# \mathscr{P}_{\sml}(\bar{A}_i,\bar{B}_j)}  (\frac{\delta \mu(i,j)}{\alpha_i \beta_j})^{-2 - \eps} \frac{\frac{X}{p^i} \frac{Y}{p^j}}{(\frac{D}{p^{\min(i,j)}})^2}. \end{equation}

On the other hand,
\begin{equation}\label{to-comp-2}  |\bar{A}_i||\bar{B}_j| = \alpha_i \beta_j |A| |B| \geq (1000)^{1+ \# \mathscr{P}_{\sml}(A,B)} \alpha_i \beta_j \delta^{-2 - \eps} \frac{XY}{D^2}   .\end{equation}
Note also that $\mathscr{P}(\bar{A}_i, \bar{B}_j) \subset \mathscr{P}(A,B) \setminus \{p\}$, and so 
\begin{equation}\label{ps-compare} \# \mathscr{P}_{\sml}(\bar{A}_i, \bar{B}_j) \leqslant \# \mathscr{P}_{\sml}(A, B) - 1_{p \leq p_0}.\end{equation}

Comparing \eqref{to-comp-1}, \eqref{to-comp-2} and \eqref{ps-compare} gives, for all $i$ and $j$, 
\begin{equation}\label{eq-22} \mu(i,j)  \leq  10^{-1_{p \leq p_0}} (\alpha_i \beta_j)^{\frac{1+\eps}{2 + \eps}} p^{-\frac{|i - j|}{2 + \eps}}, \end{equation} since $\varepsilon < 1$. 
This puts us in the situation covered by Lemma \ref{conc-lem}, with (in that lemma)
\[ q = 2 + \eps, q' = \frac{2 + \eps}{1 + \eps}, \lambda = p^{-\frac{1}{q}}, c = (\frac{1}{10})^{1_{p \leq p_0}}, x_i := \alpha_i^{1/q'}, y_j := \beta_j^{1/q'}.\] The hypotheses of the lemma are satisfied, since $p^{-\frac{1}{q}} \leqslant 2^{-\frac{1}{3}} \leqslant 4/5$. The lemma implies, first of all, that $c > \frac{1}{10}$; this immediately tells us that $p > p_0$. We conclude that there is some $k$ such that 
\begin{equation}\label{edge-defect} \sum_{|i - k| + |j - k| \geq 2} \mu(i,j) \ll  \lambda^{q + \eps} \ll  p^{-1 - \eps/3}. \end{equation}

This is precisely what is needed in Proposition \ref{minimal-counterexample-p}, taking \[\Omega_p = \bigcup_{|i - k| + |j - k| \leq 1} (\Omega \cap (A_i \times B_j)).\]

\section{Finishing the argument} \label{new-arg}

In this section we complete the proof of Theorem \ref{mainthm} by establishing Proposition \ref{not-counterexample}. That is, our task is as follows. Suppose that $A \subset [X, 2X]$, $B \subset [Y, 2Y]$, that $\Omega\ \subset A \times B$ has size $\frac{\delta}{2} |A| |B|$, and that $\gcd(a,b) \geq D$ whenever $(a,b) \in \Omega$. Suppose that there is some positive integer $N$ such that 
\begin{equation}\label{pivotal-repeat} |v_p(a/N)| + | v_p(b/N)| \leq 1 \end{equation} for all primes $p$ and for all $(a,b) \in \Omega$. We are to show that, under these assumptions, we have the bound
\begin{equation}\label{stronger} |A||B| \leq \frac{1000}{\delta^2} \frac{XY}{D^2}.\end{equation} 

Let us begin the proof. In the course of the argument it will be convenient to use a little of the language of graph theory. Thus if $a \in A$ then we write $\deg(a) := \# \{ b \in B : (a,b) \in \Omega\}$, and analogously for $b \in B$. Write $A' := \{ a \in A : \deg (a) > 0\}$ and $B' := \{b \in B : \deg(b) > 0\}$.

If $a \in A'$ then, by \eqref{pivotal-repeat}, $v_p(a/N) \in \{-1,0,1\}$ for all primes $p$. We define the \emph{defect} $a_*$ to be the product of all primes for which $v_p(a/N) \neq 0$. Now we make the crucial observation that if $(a,b) \in \Omega$ then
\begin{equation}\label{quad}   a_* b_* = \frac{ab}{\gcd(a,b)^2}. \end{equation}
To prove this, we take $p$-adic valuations. It is easily seen that 
\[ v_p(ab/\gcd(a,b)^2) = |v_p(a) - v_p(b)| = |v_p(a/N) - v_p(b/N)|, \]so we need only prove that 
\begin{equation}\label{to-prove} v_p(a_*) + v_p(b_*) = |v_p(a/N) - v_p(b/N)|\end{equation} whenever $(a,b) \in \Omega$.
This follows immediately from \eqref{pivotal-repeat}, noting that $v_p(a_*) = 1$ if $v_p(a/N) = \pm 1$ and $v_p(a_*) =0$ otherwise, and similarly for $v_p(b_*)$.

As a consequence of \eqref{quad} and our assumptions, we see that 
\begin{equation}\label{astarbstar} a_* b_* \leq \frac{4XY}{D^2}\end{equation} whenever $(a,b) \in \Omega$.
This would allow us to conclude very quickly, were it not for the fact that the map $a \mapsto a_*$ need not be injective (see Section \ref{end-rem} for some further remarks on this point). Fortunately, we have the following substitute for injectivity.

\begin{lemma}\label{count-defect}
Let $T \in \R_{> 0}$. The number of $a \in A'$ for which $a_* \leq T$ is at most $2T$.  Similarly, the number of $b \in B'$ for which $b_* \leq T$ is at most $2T$.
\end{lemma}
\begin{proof}
If $a \in A'$, write $a_+$ for the product of all primes with $v_p(a/N) = 1$, and $a_-$ for the product of all primes with $v_p(a/N) = -1$. Thus 
\begin{equation}\label{apm} a_* = a_+ a_-.\end{equation}
Since $v_p(a/N) \in \{-1,0,1\}$, we have
\begin{equation}\label{ratio-a} \frac{a_+}{a_-} = \frac{a}{N}.\end{equation} 
Since $A \subset [X, 2X]$, it follows by multiplying \eqref{apm} and \eqref{ratio-a} that if $a_* \leq T$ then
\begin{equation}\label{aplus} a_+ \leq (\frac{aT}{N})^{1/2} \leq  (\frac{2XT}{N})^{1/2}. \end{equation} Similarly, dividing \eqref{apm} by \eqref{ratio-a}, we see that if $a_* \leq T$ then
\begin{equation}\label{aminus} a_- \leq (\frac{NT}{a})^{1/2} \leq (\frac{NT}{X})^{1/2}.\end{equation}
It follows from \eqref{aplus}, \eqref{aminus} that the number of choices for the pair $(a_+, a_-)$ is at most $2T$. However, if we know $a_+, a_-$ and $N$ then we can recover $a$ uniquely, so the map $a \mapsto (a_+, a_-)$ is injective. The proof for $B'$ is the same.\end{proof}

Now we finish the argument. By a standard averaging argument there is a set $\tilde A \subset A$ with $|\tilde A| \geq \delta |A|/4$ such that $\deg(a) \geq \delta |B|/4$ for all $a \in \tilde A$. Clearly $\tilde A \subset A'$, so by Lemma \ref{count-defect} there is some $a \in \tilde A$ such that 
\begin{equation}\label{a-star-lower} a_* \geq \delta|A|/8.\end{equation}
Set $\tilde B := \{ b \in B : (a,b) \in \Omega\}$. Thus $|\tilde B| \geq \delta |B|/4$. Clearly $\tilde B \subset B'$, so by Lemma \ref{count-defect} there is some $b \in \tilde B$ such that 
\begin{equation}\label{b-star-lower} b_* \geq \delta|B|/8.\end{equation}
By construction we have $(a,b) \in \Omega$, so we have the upper bound \eqref{astarbstar}.

Comparing \eqref{astarbstar}, \eqref{a-star-lower}, \eqref{b-star-lower} immediately yields \eqref{stronger}.

\section{Further results and remarks} \label{end-rem}
Suppose that $A$ and $B$ are finite sets of square-free positive integers. In this instance, we may assume that the positive integer $N$ that satisfies \eqref{pivotal-repeat} is also square-free, and thus the map $a \mapsto a_*$ is injective, since knowing $N$ and $a_*$ determines $a$. This enables us to circumvent Lemma \ref{count-defect}, and prove the following theorem. 

\begin{theorem}\label{square-free result}
Let $Q\in [1, \infty)$ and $\delta \in (0,1]$. Suppose that $A,B$ are finite sets of square-free positive integers with the following property: for at least $\delta \vert A\vert \vert B\vert$ pairs $(a,b)\in A \times B$, $ab/\gcd(a,b)^2 \leqslant Q$. Then we have the bound \[ \vert A\vert \vert B\vert \leqslant (1000)^{1 + \# \mathscr{P}_{\sml}(A \cup B)} \delta^{-2- \varepsilon} \frac{Q}{4}.\] 
\end{theorem}
\noindent Of course this result implies the cruder bound $\vert A\vert \vert B\vert \ll \delta^{-2- \varepsilon} Q$.  It also implies Theorem \ref{mainthm}, upon taking $Q = 4XY/D^2$. 

\begin{proof}
The proof of Proposition \ref{minimal-counterexample} holds mutatis mutandis. Analysing the minimal counterexample as before, we conclude that $a_* b_* \leqslant Q$ (by analogy with \eqref{astarbstar}). Using graph theoretic language as before, there exists a set $\tilde{A} \subset A$ with $\vert \tilde{A}\vert \geqslant \delta\vert A\vert / 4$ such that $\deg(a) \geqslant \delta \vert B\vert/4$ for all $a \in \tilde{A}$. Since $a \mapsto a_*$ is injective, there is some $a \in \tilde{A}$ for which $a_* \geqslant \delta \vert A\vert/4$. Letting $\tilde{B}: = \{b \in B: (a,b) \in \Omega\}$, we have $\vert \tilde{B}\vert \geqslant \delta \vert B\vert/4$ and for all $b \in \tilde{B}$ we have $b_* \leqslant Q/a_*$. Therefore, since $b \mapsto b_*$ is injective, we have \[\frac{\delta \vert B\vert}{4} \leqslant \vert \tilde{B}\vert \leqslant \max_{b_*: b \in \tilde{B}} b^* \leqslant \frac{Q}{a_*} \leqslant \frac{4Q}{\delta \vert A\vert}.\] This rearranges to $\vert A\vert \vert B\vert \leqslant 16 \delta^{-2} Q$, which shows that the minimal counterexample is not in fact a counterexample, thus settling the theorem. 
\end{proof}

One might wonder whether the bound $\vert A\vert \vert B\vert \ll \delta^{-2 - \varepsilon} Q$ holds for general finite sets of integers $A$ and $B$ (not just for square-frees). However there is a counterexample to this assertion, even with $\delta = 1$, given by \[ A = B = \Big\{\Big( \prod\limits_{p \leqslant X} p\Big) \frac{m}{n}:  mn \leqslant X, \, \mu^2(m) = \mu^2(n) = 1, \, \gcd(m,n) = 1\Big\}.\] One may establish that for all $(a_1,a_2) \in A \times A$ one has the bound $a_1a_2/\gcd(a_1,a_2)^2 \leqslant X^2$. Yet $\vert A\vert \gg X \log X$. 

By this, one notes that the use of dyadic ranges in the proof of Lemma \ref{count-defect} was critical.

\end{document}